\documentclass{amsart}
\usepackage{amssymb}
\usepackage[all]{xy}

\title{Frobenius morphisms of noncommutative blowups}

\author{Takehiko Yasuda}
\address{Department of Mathematics and Computer Science, 
Kagoshima University, 1-21-35 Korimoto, Kagoshima 890-0065, Japan}
\email{yasuda@sci.kagoshima-u.ac.jp}

\theoremstyle{plain}
\newtheorem{thm}{Theorem}[section]

\newtheorem{prop}[thm]{Proposition}
\newtheorem{cor}[thm]{Corollary}
\newtheorem{lem}[thm]{Lemma}

 \theoremstyle{definition}
\newtheorem{defn}[thm]{Definition}
\newtheorem{ex}[thm]{Example}

\theoremstyle{remark}
\newtheorem{rem}[thm]{Remark}

\def\AA{\mathbb A}

\def\Znn{\mathbb{Z}_{\ge 0}}

\def\cA{\mathcal{A}}
\def\cB{\mathcal{B}}

\def\cE{\mathcal{E}}
\def\cF{\mathcal{F}}

\def\cL{\mathcal{L}}
\def\cM{\mathcal{M}}
\def\cN{\mathcal{N}}
\def\cO{\mathcal{O}}

\def\bR{\mathbf{R}}

\def\11{\mathbf{1}}

\def\lmodules{\text{-}\mathbf{mod}}

\DeclareMathOperator{\Spec}{Spec}

\DeclareMathOperator{\Hom}{Hom}
\DeclareMathOperator{\cHom}{\mathcal{H}\mathit{om}}

\DeclareMathOperator{\DB}{DB}

\DeclareMathOperator{\NCB}{NCB}

\DeclareMathOperator{\End}{End}
\DeclareMathOperator{\cEnd}{\mathcal{E}\mathit{nd}}

\DeclareMathOperator{\Qcoh}{Qcoh}
\DeclareMathOperator{\Coh}{Coh}

\def\PS{\mathcal{PS}}
\def\ps{\mathrm{ps}}

\begin{document}

\maketitle

\begin{abstract}
We define the Frobenius morphism of certain class of noncommutative blowups
in positive characteristic.
Thanks to a nice property of the class, the defined morphism is flat.
Therefore we say that the noncommutative
blowups in this class are Kunz regular. One of such blowups is the one associated to 
a regular Galois alteration. As a consequence of de Jong's theorem, we see that
for every variety over an algebraically closed field of positive characteristic, 
there exists a noncommutative blowup which is Kunz regular. 
We also see that  a variety with F-pure and FFRT (finite F-representation type) singularities
has a  Kunz regular noncommutative blowup which is associated to
 an iteration of the Frobenius morphism of the variety. 
\end{abstract}

\section{Introduction}\label{sec-intro}

The Frobenius morphism is arguably the most important notion in the algebraic geometry
of positive characteristic and used almost everywhere.
Concerning the singularity theory, 
Kunz's theorem is classical \cite{MR0252389}: A scheme is regular if and only if its Frobenius morphism is flat. The main aim of this article is to define the Frobenius morphism
of certain class of noncommutative blowups in positive characteristic
and to see that the defined morphism is flat. 
Here the noncommutative blowup that we mean 
is basically the same as the noncommutative crepant resolution 
in  \cite{MR2077594} and 
the noncommutative desingularization in \cite{MR1957019}
except that we remove some assumptions, especially 
the finiteness of global dimension. 

Let $k$ be a field of characteristic $p > 0$.
Recently it was found in \cite{Yasuda:arXiv:0810.1804} that if $X=\Spec R$ is from some classes of singularities over $k$, then for sufficiently large $e$,
the endomorphism ring $\End_R (R^{1/p^e})$
, whose elements are differential operators on $R^{1/p^e}$, has finite global dimension
and is regarded as a noncommutative resolution of $X$. 
This article derives from the author's attempt to know where the regularity
of $\End_R (R^{1/p^e})$ comes from
and to show its regularity for a broader class of singularities. 
However the regularity which we will consider in this article is the flatness of Frobenius 
rather than the finiteness of global dimension. It is because 
the former seems to the author simpler and compatible with $\End_R (R^{1/p^e})$.
We however consider not only noncommutative blowups of the form $\End_R (R^{1/p^e})$.

Let $X,Y$ be integral normal Noetherian schemes over $k$ 
with finite Frobenius morphisms and $f:Y \to X$ a finite dominant morphism. 
We associate to $f$  a noncommutative blowup $\NCB(Y/X)$, which 
is the pair of the endomorphism ring $\cEnd_{\cO_X}(\cO_Y)$ 
and the left $\cEnd_{\cO_X}(\cO_Y)$-module $\cO_Y$.
(More generally we will consider the noncommutative blowup associated to a coherent sheaf. However the examples in which we are interested are associated to finite morphisms of schemes.)
 Also we regard this
as the category of left $\cEnd_{\cO_X}(\cO_Y)$-modules with 
the distinguished object $\cO_Y$.
Let $Y_e \to Y$ be the $e$-times
iteration of the Frobenius morphism.
We say that $f$ is \emph{F-steady} if for every $e \ge 0$, the structure sheaves of $Y_e$
and $Y$ locally have, as $\cO_X$-modules, the same summands (for details, see Section \ref{sec-Frobenius}). For instance, if $Y$ is regular, then $f$ is F-steady.
Given an F-steady morphism $Y \to X$, we
 define the Frobenius morphism of $\NCB(Y/X)$, which is flat by construction. 
Hence we say that $\NCB(Y/X)$ is \emph{Kunz regular}.

If $k$ is algebraically closed, 
 from de Jong's theorem \cite{MR1423020},
every $k$-variety $X$ admits a Galois alteration $Y \to X$ with $Y$ regular.
It uniquely factors as $Y \to \bar Y \to X$
such that $\bar Y$ is a normal variety, $Y \to \bar Y$
is finite and $\bar Y \to X$ is a modification.
Then the associated noncommutative blowup $\NCB(Y/X)=\NCB(Y/\bar Y)$
is Kunz regular. Thus every variety admits a noncommutative blowup which is Kunz regular (Corollary \ref{cor-main}).

Another interesting example of noncommutative blowups is the one associated
to an iterated Frobenius morphism $X_e \to X$ of a normal scheme $X$. 
In the affine case, this
corresponds to the above-mentioned ring $\End_R(R^{1/p^e})$. 
If $X$ has only F-pure and FFRT (finite F-representation type) singularities,
then for sufficiently large $e$, $X_e\to X$ is F-steady 
and the associated noncommutative blowup $\NCB(X_e/X)$ is Kunz regular (see Section \ref{sec-DB}). 
The FFRT singularity was introduced in \cite{MR1444312}
and proved to have $D$-module theoretic nice properties \cite{MR1444312,MR2407232}.  
Our result is yet another such property.

\subsection{Convention}

Throughout the paper, we work over a fixed base field $k$ unless otherwise noted.
We mean by a \emph{scheme} a separated Noetherian scheme over $k$.
In Sections \ref{sec-Frobenius}, \ref{sec-DB} and \ref{sec-compare},
we additionally assume that $k$ has characteristic $p>0$ and that
every scheme is F-finite, that is,  the Frobenius morphism is finite.
If $f : Y \to X$ is an affine morphism of schemes
and $\cM$ is a quasi-coherent sheaf on $Y$, 
then by abuse of notation, we denote the push-forward  $f_* \cM$ again by
$\cM$.

\subsection{Acknowledgment}

This work was supported by Grant-in-Aid for Young Scientists 
(20840036) from JSPS.

\section{Noncommutative schemes}\label{sec-nc-scheme}

\subsection{Pseudo-schemes}

Following  \cite[page 235]{MR1304753},
we first define a category.

\begin{defn}
A \emph{pseudo-scheme} is the pair $(\cA,M)$ 
of a $k$-linear abelian category $\cA$ and an object $M \in \cA$. 
 A \emph{morphism} $f: (\cA,M) \to (\cB ,N )$ of pseudo-schemes 
is the equivalence class of  pairs $(f^*, \theta)$ of a $k$-linear functor $f^*: \cB \to \cA$
which admits a right adjoint $f_* : \cA \to \cB$ and 
an isomorphism $\theta : f^*M \cong N$. Here two such pairs $(f^*, \theta)$ and  $((f^*)', \theta')$ are equivalent if there is an isomorphism $f^* \cong (f^*)'$ which is compatible
with $\theta$ and $\theta '$. The composition of morphisms is defined in the obvious way.
We denote the category of pseudo-schemes by $\PS$.
A morphism $f$ is said to be \emph{flat} if its pull-back functor $f^*$ is exact. 
\end{defn}

For a scheme $X$, we denote by $\Qcoh(X)$ the category of quasi-coherent sheaves on $X$.
We have a natural functor
\[
 (\text{scheme}) \to \PS, \  X \mapsto X^\ps:= (\Qcoh(X),\cO_X) .
\]
From a theorem of Gabriel \cite{MR0232821}, we can reconstruct $X$ from $X^\ps$
(which was generalized to the non-Noetherian case by Rosenberg \cite{MR1622759}):

\begin{thm}[Reconstruction of schemes]
If $X^\ps \cong Y^\ps$, then $X\cong Y$.
\end{thm}

We can also reconstruct morphisms:

\begin{prop}[Reconstruction of morphisms]
The functor $X \mapsto X^\ps$ is faithful.
\end{prop}

\begin{proof}
Suppose that $f :Y \to X=  \Spec A$ be a morphism of schemes with $X$ affine. 
Then $f^\ps$ determines a $k$-algebra map
 $A = \End (A) \to  \Gamma(\cO_Y) = \End (\cO_Y)$ and so determines $f$.

Next suppose that $f:Y \to X$ be an arbitrary morphism of schemes. Then applying $f_*$ to 
the structure sheaves of integral closed subschemes of $Y$, we see that $f^\ps$ uniquely determines  $f$ as the map of sets. For each affine open subset
 $ \iota : U \hookrightarrow X$, applying $f^*$ to the sheaves
$\iota_* \cM$, $\cM \in \Qcoh(U)$, we see that $f^\ps$ uniquely determines
the scheme morphism $f|_{f^{-1}(U)} :f^{-1}(U) \to U$. As a consequence, $f^\ps$ uniquely determines $f$. Hence the functor  is faithful. 
\end{proof}

The above results allow us to identify a scheme $X$ (resp.\ a scheme morphism $f$)
with $X^\ps$ (resp.\ $f^\ps$).

\subsection{Noncommutative schemes}

\begin{defn}
Let $Z$ be a scheme.
A \emph{finite NC (noncommutative) scheme over $Z$} 
is the pair $  X=(\cA , \cM )$ of 
a coherent sheaf $\cA$ of $\cO_X$-algebras and a coherent sheaf  $\cM$ of left 
$\cA$-module.
We denote by $\Qcoh(X)=\Qcoh(\cA)$ the category of quasi-coherent left 
$\cA$-modules and set $X^\ps := (\Qcoh(X),\cM)$.  
Like a scheme,   we often identify $X$ and $X^\ps$.

A \emph{morphism} $X =( \cA , \cM) \to X' =(\cA ', \cM')$ of finite NC schemes over  $Z$ is
a morphism $X^\ps \to (X')^\ps$
defined by the functor 
\[
 \cN \otimes_{\cA'} - : \Qcoh(X') \to \Qcoh(X)
 \]
for some coherent sheaf  $\cN$ of $ (\cA , \cA') $-bimodules and an isomorphism $\cN \otimes_{\cA'} \cM' \cong  \cM$. Note that the functor has the right adjoint $\cHom _{\cA} ( \cN,-)$
and indeed defines a morphism in $\PS$.
\end{defn}

We do not construct the correct category of finite NC schemes over different schemes in this article.
Instead we will  work in 
the ambient category $\PS$.

\section{Alterations and noncommutative blowups}\label{sec-alteration}

\begin{defn} 
A morphism $Y \to X$ of integral schemes is called an \emph{alteration}
(resp.\ \emph{modification}) if it is generically finite, dominant and proper (resp.\ birational and proper). 
An alteration $Y \to X$ is said to be \emph{normal} (resp.\ \emph{regular}) 
if $Y$ is so. 
A \emph{finite-birational factorization} of a normal alteration $Y \to X$   is 
a factorization of $Y \to X$ into  a finite and dominant morphism $Y \to \bar Y$
and a modification $\bar Y \to X$ with $\bar Y$ normal.
(This is clearly unique up to isomorphism if exist.) 
A normal alteration is said to be \emph{factorizable} 
if it admits a finite-birational factorization.
An alteration $f : Y \to X$ is said to be \emph{Galois} if 
there exists a finite group $G$ of automorphisms of $Y$
such that if we give to $X$ the trivial $G$-action, then  $f$ is $G$-equivariant
and the field extension $K(Y)^G /K(X)$ is  purely inseparable. 
\end{defn}

\begin{lem}
Let $f:Y \to X$ be a normal Galois alteration. 
Suppose that the quotient algebraic space $Y/G$ is a scheme.
In the case where $k$ has positive characteristic, we suppose that $Y/G$ is F-finite.
Then $f$ is factorizable.
\end{lem}

\begin{proof}
In characteristic $0$, the natural morphism $Y/G \to X$ is birational,
hence $f$ is factorizable. Let us suppose that $k$ has characteristic $p>0$.
We take $e \in \Znn$ such that  $(K(Y)^G)^{p^e} \subset K(X)$.
Let $Y/G$ be the quotient variety and $Y/G \to (Y/G)^{e}$ the morphism
corresponding to the inclusion $\cO_{Y/G}^{p^e} \hookrightarrow \cO_{Y/G}$ of sheaves,
though this is not a morphism of $k$-schemes unless $k$ is perfect. 
Let $\bar Y$ be the normalization of $(Y/G)^e$ in $K(X)$.
Then we claim that $f$ factorizes through $\bar Y$.
To see this, we 
take affine open coverings $Y = \bigcup \Spec S_i$ and $X = \bigcup \Spec R_i$
such that for each $i$, $f(\Spec S_i) \subset \Spec R_i$
and $\Spec S_i$ is stable under the $G$-action. 
Set $\bar S_i := S_i \cap K(X)$. Then $\bar Y = \bigcup \Spec \bar S_i$.
Since $S_i \supset \bar S_i \supset R_i$,  the claim holds. 
\end{proof}

\begin{defn}\label{def-NC-blowup}
For a torsion-free coherent sheaf $\cM$ on an integral scheme $X$,
we write $\cE_{\cM/X} := \cEnd_{\cO_X}(\cM)$. 
We define the \emph{NC blowup of $X$ associated to $\cM$}, $\NCB(\cM/X)$,
to be the finite NC scheme $(\cE_{\cM/X},\cM )$ over $X$.

For a dominant finite morphism $f:Y \to X$ of integral schemes,
we put $\cE_{Y/X} := \cE_{\cO_{Y} /X}$ and 
$\NCB(Y/X):= \NCB(\cO_{Y}/X)$. 
Since $\cO_Y$ is a subring of $\cE_{Y/X}$,
we have the induced functor 
$\Qcoh(\cE_{Y/X}) \to  \Qcoh(\cO_Y)$,
which is identical to $_{\cO_Y} \cE_{Y/X} \otimes_{\cE_{Y/X}} -$. 
We call the corresponding morphism $Y \to \NCB(Y/X)$ the \emph{coforgetful morphism}.
This is obviously flat.
We define the \emph{projection} $\NCB(Y/X) \to X$ 
by $_{\cE_{Y/X}} \cO_Y \otimes_{\cO_X} -$.
The composite morphism
\begin{equation*}\label{factorization-noncommutative}
Y \to \NCB (Y/X)\to X.
\end{equation*}
is exactly the original morphism $Y \to X$. 
\end{defn}

\begin{defn}
For a factorizable normal alteration $Y \to X$, 
if  $Y \to \bar Y \to X$ is the finite-birational factorization, then 
we define the \emph{associated NC blowup}, $\NCB(Y/X)$, to be $\NCB(Y/\bar Y)$.
\end{defn}

\begin{rem}
The normality assumption in the above definition is not really necessary, but just for simplicity. 
\end{rem}

Every factorizable normal alteration $Y \to X$
factors also as $Y \to \NCB(Y/X) \to X$.

\section{Frobenius morphisms}\label{sec-Frobenius}

In this section, we shall define the Frobenius morphism for some class of
noncommutative blowups.

From now on, we suppose that the base field $k$ has characteristic $p>0$.
We also suppose that every scheme has the \emph{finite} Frobenius morphism.

\subsection{Equivalent modules}

\begin{defn}
Let $R$ be a commutative complete local Noetherian ring.
Then every finitely generated $R$-module is uniquely the direct sum of finitely many
indecomposable $R$-modules. 
We say that 
$R$-modules $M$ and $N$ are \emph{equivalent}
if $M\cong \bigoplus_{i} L_{i}^{\oplus a_{i}}$ 
and $N \cong \bigoplus_{i} L_{i}^{\oplus b_{i}}$
for some indecomposable $R$-modules $L_{i}$ and positive integers $a_{i}$ and $b_{i}$.

We say that  coherent sheaves $\cM$ and $\cN$ on a scheme $X$ are \emph{equivalent}
if for every point $x \in X$, the complete stalks $\hat \cM_{x}$
and $\hat \cN_{x}$ are equivalent $\hat \cO_{X,x}$-modules.
\end{defn}

Given  coherent sheaves $\cM$ and $\cN$ on a scheme $X$.
We think of $\cM$ as an $(\cEnd(\cM),\cO_X)$-bimodule
and similarly for $\cN$. Then 
$\cHom(\cM,\cN) = \cHom_{\cO_X} (\cM,\cN)$ is an $(\cEnd (\cN),\cEnd (\cM) )$-bimodule.

\begin{lem}\label{lem-isom-hom-tensor}
Let $\cL$, $\cM$ and $\cN$ be coherent sheaves on
$X$ which are mutually  equivalent.
\begin{enumerate}
\item We have a natural isomorphism of $(\cEnd(\cN),\cEnd(\cL))$-bimodules
\[
\cHom (\cM,\cN) \otimes_{\cEnd (\cM)} \cHom (\cL,\cM) \cong \cHom (\cL,\cN) .
\]
In particular
\[
\cHom (\cM,\cN) \otimes_{\cEnd (\cM)} \cHom (\cN,\cM) \cong \cEnd (\cN) .
\]
Hence the functors 
\begin{align*}
\cHom (\cM,\cN) \otimes_{\cEnd (\cM)} - : \Qcoh( \cEnd (\cM)) \to \Qcoh( \cEnd (\cN)) \\
\cHom (\cN,\cM) \otimes_{\cEnd (\cN)} - : \Qcoh( \cEnd (\cN)) \to \Qcoh( \cEnd (\cM))
\end{align*}
are equivalences which are inverses to each other.
\item
We have a natural isomorphism of $(\cEnd(\cN),\cO_X)$-bimodules
\[
\cHom (\cM,\cN) \otimes_{\cEnd (\cM)} \cM \cong \cN .
\]
\end{enumerate}
\end{lem}

\begin{proof}
These are well-known to the specialists.
\begin{enumerate}
\item There exists a natural morphism 
\[
\cHom (\cM,\cN) \otimes_{\cEnd (\cM)} \cHom (\cL,\cM) \to \cHom (\cL,\cN) .
\]
It is easy to see that the morphism is an isomorphism after the completion
at each point of $X$. Hence the morphism is an isomorphism. 
\item
The proof is similar to the above.
\end{enumerate}
\end{proof}

\subsection{F-steady modules and Frobenius morphisms}

For a scheme $X$, we write the $e$-iterated $k$-linear Frobenius
as
\[
F^e=F^e_X :X_e  \to X .
\]
Sometimes we simply call this the \emph{$e$-th Frobenius of $X$}.
A key observation is that the morphism $F^e$ factors as 
$X_e \to \NCB(X_e/X) \to X$ (see Definition \ref{def-NC-blowup}).

\begin{defn}
Let $X$ be an integral normal scheme and $\cM$ a reflexive coherent sheaf
(that is, $\cM^{\vee \vee} \cong \cM$). 
Denote by $\cM_{e}$ the sheaf on $X_{e}$ corresponding to $\cM$
via the obvious identification $X_{e} =X$. Then $ \cM_{e}$
is identical as an $\cO_{X}$-module to the push-forward of $\cM$ by the $e$-iterated \emph{absolute}
Frobenius.
We say that $\cM$ is  \emph{F-steady}
if for every $e$, $\cM$ and $\cM_e$ are equivalent $\cO_X$-modules.

For   a finite dominant morphism $f: Y \to X$ of integral normal schemes,
we say that $f$ is \emph{F-steady} if $\cO_{Y}$ is an F-steady $\cO_{X}$-module. 
\end{defn}

\begin{ex}
 If $Y$ is regular, then $f$ is F-steady. Indeed being flat over $\cO_Y$, 
 $\cO_{Y_e}$ is locally isomorphic to $\cO_Y^{\oplus r}$, $r>0$,
 as an $\cO_Y$-module and hence also as an $\cO_X$-module. 
\end{ex}

From Lemma \ref{lem-isom-hom-tensor},
for an $F$-steady sheaf $\cM$,
we have an isomorphism 
\[
\NCB(\cM_e/X) \cong \NCB(\cM_{e'}/X) , \  e,e' \ge 0. 
\]
We also define a morphism 
\[
 \NCB(\cM_e/X_e) \to \NCB(\cM_e/X)  ,
 \]
 which we call the \emph{coforgetful morphism},
  as follows.
We think of $\cO_X$ as a subring of $\cO_{X_e} = \cO_{X}^{1/p^{e}}$
in the obvious way. Then  $\cE_{\cM_e/X_e}$ is a subring of $ \cE_{\cM_e/X}$.
Hence 
we have a natural morphism $\NCB(\cM_e/X_e) \to \NCB(\cM_e/X)$
defined by $_{\cE_{\cM_e/X_e}}\cE_{\cM_e/X} \otimes_{  \cE_{\cM_e/X}}-$.

\begin{defn}
Let $\cM$ be  an F-steady sheaf on $X$.
We define the \emph{$e$-th Frobenius} of $\NCB(\cM/X)$
to be the composite 
\[
F^e=F^e_{\NCB(\cM/X)} : \NCB(\cM_e/X_e) \xrightarrow{\textnormal{cofor.}} \NCB(\cM_e/X) \xrightarrow{\sim} \NCB(\cM/X) .
\]
By construction, this is flat, which we call the \emph{Kunz regularity} of $ \NCB(\cM/X)$.
(Recall that from Kunz \cite{MR0252389}, a scheme is regular if and only if its Frobenius morphisms are flat.)
 The morphism  is also directly defined by the functor 
 \[
 _{\cE_{\cM_e/X_e}} \cHom_{\cO_X}(\cM , \cM_{e}) \otimes_{\cE_{\cM/X}} - .
 \] 
\end{defn}

\begin{cor}\label{cor-main}
Suppose that $k$ is algebraically closed
and that $X$ is an arbitrary $k$-variety. Take a regular Galois alteration $Y\to X$ with $Y$ quasi-projective.
(Such an alteration exists from de Jong's theorem \cite[7.3]{MR1423020}).
Then the associated NC blowup $\NCB(Y/X)$  is Kunz regular.
\end{cor}

\begin{proof} We first note that $ Y \to X$ is factorizable.  
Let $Y \to \bar Y \to X$ be the finite-birational factorization.
Since $Y$ is regular, the morphism $Y \to \bar Y$ is F-steady. Hence  $ \NCB(Y/\bar Y) = \NCB(Y/X)$
 is Kunz regular.
\end{proof}

\begin{rem}
Bondal and Orlov \cite{MR1957019} conjectured the following: 
Let $Y \to X$ be a finite morphism of varieties such that $X$ has canonical singularities
and $Y$ is regular. Then the derived category of $\cE_{Y/X}$-modules is a minimal
categorical desingularization. 
Their conjecture and the above corollary seem somehow related.   
\end{rem}

\subsection{Compatibilities of Frobenius morphisms}

In this subsection,  to justify our definition of the Frobenius morphism,
we show some compatibilities of it (see also Section \ref{sec-compare}). 
We suppose that $\cM$ is an F-steady reflexive
 coherent sheaf on an integral normal scheme $X$.

\begin{prop}
The diagram 
\[
\xymatrix{
\NCB(\cM_e/X_e) \ar[r]^{F^e} \ar[d]_{\textnormal{proj.}} & \NCB(\cM/X) \ar[d]^{\textnormal{proj.}} \\
X_{e} \ar[r]_{F^e} & X
}
\]
is commutative. 
\end{prop}

\begin{proof}
From Lemma \ref{lem-isom-hom-tensor}, we have isomorphisms
of $(\cE_{\cM_e/X_e},\cO_X)$-bimodules
\[
     \cM_{e}    \otimes_{\cO_{X_e}}  \cO_{X_e} \cong  \cM_{e}  \cong      
       \cHom_{\cO_X}( \cM,\cM_{e})  
     \otimes_{\cE_{\cM/X}} \cM ,
\]
which proves the proposition.
\end{proof}

\begin{lem}\label{lem-commutative}
For $e'\ge e \ge 0$,  the diagram
\[
\xymatrix{
\NCB(\cM_{e'}/X_e)  \ar[r]^{\textnormal{cofor.}} \ar[d]_\wr & \NCB(\cM_{e'}/X) \ar[d]^\wr \\
\NCB(\cM_e/X_e)  \ar[r]_{\textnormal{cofor.}}  & \NCB(\cM_e/X)
}
\]
 is commutative.
\end{lem}

\begin{proof}
There exists a natural morphism 
\begin{align*}
\alpha : \cHom_{\cO_{X_e}}(\cM_e,\cM_{e'}) \otimes _{\cE_{\cM_e/X_e} } \cE_{\cM_e/X} & \to \cHom_{\cO_{X}}(\cM_{e},\cM_{e'})  \\ 
\phi \otimes \psi & \mapsto \phi \circ \psi . 
\end{align*}
We claim that this is an isomorphism, which proves the lemma.

Let $U \subset X$ be an open subset such that
$X \setminus U$ has codimension $\ge 2$
and $\cM|_{U}$ is locally free.
Then locally on $U$,   
we have an isomorphism of $\cO_{X_e}$-modules 
$\cM_{e'} \cong \cM_{e} ^{\oplus r}$ for some $r$.
Hence locally on $U$, the source and target of $\alpha$ are both isomorphic to $\cE_{\cM_e/X}^{\oplus r}$. 
It is now easy to see that $\alpha$ is an isomorphism over $U$.

Moreover both hand sides are flat right $\cE_{\cM_e/X}$-modules
and  hence locally isomorphic to direct summands of 
$ \cE_{\cM_e/X}^{\oplus l} $ for some $l$.
From the normality assumption, $ \cE_{\cM_e/X}^{\oplus l} $ is a reflexive $\cO_X$-module 
(see \cite[Prop.\ 1.6]{MR597077})
and so are its direct summands. So $\alpha$ is an isomorphism all over $X$. 
We have proved the claim and the lemma. 
\end{proof}

\begin{cor}\label{cor-commutative-frob}
For $e,e' \ge 0 $,  the diagram
\[
\xymatrix{
\NCB(\cM_{e+e'}/X_e)  \ar[r]^{F^e} \ar[d]_\wr &  \NCB(\cM_{e'}/X) \ar[d]^\wr \\
\NCB(\cM_e/X_e)  \ar[r]_{F^e}  & \NCB(\cM/X) 
}
\]
is commutative. 
\end{cor}

\begin{proof}
If $e' \ge e$, then from Lemmas \ref{lem-isom-hom-tensor} and  \ref{lem-commutative}, the diagram
\[
\xymatrix{ 
&  \NCB(\cM_{e+e'}/X) \ar[d] \ar@/^50pt/[ddd]\\
\NCB(\cM_{e+e'}/X_e) \ar[ur] \ar[d] \ar@/_40pt/[dd] & \NCB(\cM_{e'}/X) \ar[d] \\
\NCB(\cM_{e'}/X_e) \ar[ur] \ar[d] &  \NCB( \cM_e/X)\ar[d] \\
\NCB(\cM_e/X_e) \ar[ur]  & \NCB(\cM/X)
}
\]
is commutative.
Now the corollary follows from our definition of the Frobenius morphism.

If $e' < e$, then similarly the diagram
\[
\xymatrix{ 
&  \NCB(\cM_{e+e'}/X) \ar[d] \ar@/^50pt/[ddd]\\
\NCB(\cM_{e+e'}/X_e) \ar[ur] \ar[d]  \ar@/_40pt/[dd] & \NCB(\cM_{e}/X) \ar[d] \\
\NCB(\cM_{e}/X_e) \ar[ur] \ar[d] &  \NCB( \cM_{e'}/X)\ar[d] \\
\NCB(\cM_{e'}/X_e) \ar[ur]  & \NCB(\cM/X)
}
\]
is commutative and the corollary follows.

\end{proof}

\begin{cor}
For $e,e' \ge 0$, the diagram
\[
\xymatrix{
\NCB(\cM_{e+e'}/X_{e+e'})  \ar[r]^{F^{e'}} \ar[dr]_{F^{e+e'}} & \NCB(\cM_e/X_e) \ar[d]^{F^e} \\
\  & \NCB(\cM/X) 
}
\]
is commutative. Namely we have $F^{e+e'} = F^e \circ F^{e'}$.
In particular, the $e$-th Frobenius of $\NCB(\cM/X)$ is the $e$-iterate of the first Frobenius. 
\end{cor}

\begin{proof}
This follows from the commutativity of the diagram
\[
\xymatrix{ 
&& \NCB(\cM_{e+e'}/X) \ar[d] \\
& \NCB(\cM_{e+e'}/X_e) \ar[ur] \ar[d] &  \NCB(\cM_e/X) \ar[d]  \\
\NCB(\cM_{e+e'}/X_{e+e'}) \ar[ur]  & \NCB(\cM_e/X_e) \ar[ur] & \NCB(\cM/X) .
}
\]
\end{proof}

\begin{prop}
Let $f : Y \to X$  be a finite dominant morphism with $Y$ is regular. 
Then the diagram 
\[
\xymatrix{
Y_e \ar[r]^{F^e} \ar[d]_{\textnormal{cofor.}}  & Y \ar[d]^{\textnormal{cofor.}} \\
\NCB(Y_e/X_e) \ar[r]_{F^e} & \NCB(Y/X) 
}
\]
is commutative. 
\end{prop}

\begin{proof}
Because $\cO_{Y_e}$ is a locally free $\cO_Y$-module,
the canonical map 
\[
 \cO_{Y_e} \otimes_{\cO_Y} \cE_{Y/X} \to \cHom_{\cO_X}(\cO_Y, \cO_{Y_e}) .
\]
is an isomorphism, which proves the proposition. 
\end{proof}

\section{D-blowups}\label{sec-DB}

Among NC blowups, especially interesting are
the ones associated to Frobenius morphisms of schemes. 

\begin{defn}
For an integral scheme $X$,  
we define the \emph{$e$-th $D$-blowup} of $X$ as
$\DB_e(X):= \NCB(X_e/X)$.
\end{defn}

\begin{rem}
The $D$-blowup can be regarded as the noncommutative counterpart of the F-blowup (see \cite{Yasuda:arXiv:0810.1804}).
\end{rem}

\begin{defn}[Hochster--Roberts \cite{MR0417172}]
Let $X=\Spec R$ be an integral scheme. 
We say that $R$ and $X$ are \emph{F-pure} if $R \hookrightarrow R_e$
splits as an $R$-module map.
\end{defn}
 
\begin{defn}[Smith--Van den Bergh \cite{MR1444312}]
Suppose that $R$ is a complete local Noetherian domain so that the Krull-Schmidt
decomposition holds for finitely generated $R$-modules.
Then $R$ and $\Spec R$ are said to be \emph{FFRT (finite F-representation type)}
if there are finitely many indecomposable $R$-modules $M_i$, $i=1, \dots, n$,
such that for any $e$, $R_e$ is isomorphic to 
$ \bigoplus_{i=1}^n M_i ^{\oplus r_i}$, $r_i \ge0$, as an $R$-module. 
 \end{defn}
 
\begin{prop}
Let $R$ be a complete local Noetherian normal domain.
Suppose that $X=\Spec R$ is F-pure and FFRT.
Then for sufficiently large $e$, the Frobenius morphism 
$F^e_X: X_e \to X$ is F-steady.
\end{prop}

\begin{proof}
Let $M_i$, $i=1,\dots,n$, be the \emph{irredundant} set of indecomposable modules as in the 
above definition. Then there exists $e_0$ such that for every  $e \ge e_0$,
$R_e$ is isomorphic, as an $R$-module, to $\bigoplus_{i=1}^n M_i ^{\oplus r_i}$, $r_i >0$.  
Hence $X_e \to X$ is F-steady.
\end{proof}

As a corollary, we obtain the following. 

\begin{cor}
Let $X$ be an integral normal scheme with F-pure and FFRT singularities.
Namely the completion of every local ring of $X$ is F-pure and FFRT.
Then for sufficiently large $e$, $\DB_e(X)$ is Kunz regular. 
\end{cor}

\begin{ex}
Normal toric singularities and tame quotient singularities 
are F-pure and FFRT. See \cite{MR2407232} for other examples. 
\end{ex}

\section{Comparing Frobenius morphisms of commutative and noncommutative blowups}\label{sec-compare}

Let $X=\Spec R$ be an integral normal affine scheme
and $M$ a finitely generated reflexive  $R$-module such that
the associated sheaf $\tilde M$  is F-steady.  
Let $g: Z \to X$ be a modification which is a flattening of $\tilde M$, that is,
$\cM := g^*  \tilde M /tors$ is locally free. 

\begin{lem}
We have $\Gamma( \cM)=M$.
\end{lem}

\begin{proof}
There exists an open subset $U \subset X$
such that $X \setminus U$ has codimension $\ge 2$ and $\tilde M$ is locally free on $U$.
Since $X$ is normal, from Zariski's main theorem, $(g_* \cM)|_U=g_*g^*(\tilde M|_U)=\tilde M|_U$.
It follows that
 the natural morphism $\tilde M \to g_* \cM$ is an injection into a torsion-free sheaf
which is an isomorphism over $U$. 
Since $\tilde M$ is reflexive, this is an isomorphism.
Therefore we have
\[
\Gamma(\cM) = \Gamma(g_* \cM) = \Gamma(\tilde M) =M.
\]
\end{proof}

Set   $E := \End_R(M)$ and $\cE := \cEnd_{\cO_Z} (\cM)$.
Then from the preceding lemma, $E = \Gamma(\cE)$.
Since $\cM$ is locally free,  the projection
\[
h : \NCB(\cM/Z) \to Z,
\]
which is defined by $ \cM \otimes_{\cO_Z} - $, is an isomorphism.

For $\cF \in \Qcoh(\cE) = \Qcoh(\NCB(\cM/Z))$,  $\Gamma(\cF)$ 
is a left $E$-module. Thus we have a left exact functor
\[
\Phi : \Qcoh(Z) \to E \lmodules, \ \cF \mapsto \Gamma(  h^*\cF) .
\]
Put $E_{e}:= \End_{R_{e}}(M_{e})$.
Similarly we have a functor
\[
\Phi_e : \Qcoh(Z_e) \to  E_{e}\lmodules.
\]

\begin{prop}
The diagram
\[
\xymatrix{
\Qcoh(Z) \ar[r]^\Phi \ar[d]_ {(F^e)^*}&  E \lmodules \ar[d]^{(F^e)^*} \\
\Qcoh(Z_e) \ar[r]_{\Phi_e} & E_{e}\lmodules
}
\]
is commutative up to isomorphism of functors.
\end{prop}

\begin{proof}
We claim that for $\cF \in \Qcoh(\cE)$,  there exists a natural isomorphism
\[
\Hom_R(M,M_e)  \otimes_E \Gamma( \cF ) 
 \cong \Gamma(  \cHom_{\cO_Z}(\cM,\cM_e) \otimes_{\cE}  \cF ) .
\]
Obviously there exists a natural morphism from the left-hand side to the right-hand side.
Since the claim is local on $X$, to show this, we may suppose that $R$
is a complete local ring. 
Then the claim easily follows from the definition of equivalent modules. 

We have natural isomorphisms
\begin{align*}
&((F^e)^* \circ \Phi) (\cF) &  \\
&=   \Hom_R(M,M_e)  \otimes_E \Gamma( \cM \otimes_{\cO_Z} \cF ) &\\
 &\cong \Gamma(  \cHom_{\cO_Z}(\cM,\cM_e) \otimes_{\cE}  \cM \otimes_{\cO_Z} \cF ) &\\
 &\cong \Gamma(  \cM_e \otimes_{\cO_Z} \cF )  & \quad \text{(Lemma \ref{lem-isom-hom-tensor})}\\
 &\cong \Gamma(  \cM_e \otimes_{\cO_{Z_e}} \cO_{Z_e} \otimes_{\cO_Z} \cF ) & \\
& \cong \Gamma(  \cM_e \otimes_{\cO_{Z_e}} (F^e)^*  \cF ) &\\
& \cong ( \Phi_e \circ (F^e)^* )( \cF). &  
\end{align*}
Thus the proposition holds. 
\end{proof}

We have the right derived functor  of $\Phi$
\[
\bR \Phi : D^+ (\Qcoh(Z)) \to D^+(E\lmodules).
\]
Similarly for $\Phi_e$.

\begin{cor}
The diagram
\[
\xymatrix{
D^+(\Qcoh(Z)) \ar[r]^(.4){\bR \Phi} \ar[d]_ {(F^e)^*}& D^+(E\lmodules) \ar[d]^{(F^e)^*} \\
D^+(\Qcoh(Z_e)) \ar[r]_(.4){\bR \Phi_e} & D^+(E_{e}\lmodules)
}
\]
is commutative up to isomorphism of functors.
\end{cor}

\begin{proof}
We have 
\[
 (F^e)^* \circ \bR \Phi \cong \bR((F^e)^* \circ \Phi) \cong \bR(\Phi \circ (F^e)^*)
 \cong (\bR\Phi )  \circ (F^e)^*.
\]
\end{proof}

The functor $\bR \Phi$ maps $D^b (\Coh(Z))$ into  $D^b (E\lmodules^{\mathrm{fg}})$.
Here $\Coh (Z) $ denotes the category of coherent sheaves
and $E\lmodules^{\mathrm{fg}}$ that of finitely generated left $E$-modules.
As shown in \cite{MR740077,MR1752785,MR1824990,MR2077594}, 
in some situations, the functor 
\[
\bR \Phi : D^b (\Coh(Z)) \to D^b(E\lmodules^{\mathrm{fg}})
\]
is an equivalence, a kind of Fourier-Mukai transform. Then through this equivalence, 
the Frobenius morphisms on both hand sides correspond to each other
at the level of derived category.

\begin{ex}
Let $G \subset SL_d(k)$ be a small finite subgroup of order prime to $p$ with $d=2,3$.
Set
$R : =k[x_1,\dots,x_d]^G$ and $X:= \Spec R$. 
Let $Y$ be either $\AA^d_k$ or $X_e$ for $e \gg 0$,
and let $Z$ be the universal flattening of $Y \to X$, which is 
isomorphic to the $G$-Hilbert scheme of Ito--Nakamura \cite{MR1420598}
(for the case $Y=X_e$, see \cite{Yasuda:arXiv:0810.1804,Yasuda:math0706.2700}).
If we put $M$ to be the coordinate ring of $Y$, then   the above functor is an equivalence (for instance, see \cite{Yasuda:arXiv:0810.1804,MR2077594}). 
\end{ex}

\bibliographystyle{plain}
\bibliography{/Users/highernash/mathematics/mybib.bib}

\end{document}